\documentclass[letterpaper]{article}
\usepackage{amsmath}
\usepackage{amsfonts}
\usepackage{amssymb}
\usepackage{amsthm}
\usepackage{cite}
\usepackage{graphicx}

\newcommand{\R}{\mathbb{R}}
\newcommand{\ddt}{\frac{\partial}{\partial t}}
\newcommand{\tr}{\operatorname{trace}}

\newtheorem{thm}{Theorem}

\title{Doubling a cube with positive scalar curvature}
\author{Haggai Nuchi}
\date{}

\begin{document}
\maketitle

\begin{abstract}
  In a 2013 paper (published in 2014 as \cite{gromov2014dirac}), Gromov proves that if smooth Riemannian metrics $g_i$ converge to a smooth Riemannian metric $g$ uniformly, and $g_i$ have scalar curvature uniformly bounded below, then $g$ shares the same scalar curvature lower bound. In some places in the paper, the proofs are only sketched. In this paper we explain one of those sketched steps in detail. Specifically, we prove that a cube (the product $[0,1]^n$) cannot have a Riemannian metric with positive scalar curvature, such that the faces are mean convex, and such that the dihedral angles along the edges are all acute.
  
  The proof is accomplished by taking such a metric, and producing from it a metric on the $n$--torus with positive scalar curvature, contradicting the Geroch conjecture.
\end{abstract}

\section{Introduction}\label{Sec:Intro}
  Gromov and Lawson proved in 1980 \cite{gromov1980spin} that if $X$ is a Riemannian manifold with nonnegative scalar curvature and mean convex boundary, then the double of $X$ (two copies of $X$ glued together along their common boundary) has a metric of nonnegative scalar curvature. Their method involved modifying the metric near the boundary in such a way that the boundary becomes totally geodesic, while keeping nonnegative scalar curvature. In Gromov's 2013 paper, he claims without proof that this procedure can be adapted to more general polyhedral domains, ie domains whose boundaries are not smooth manifolds but rather resemble polyhedra, and where instead of doubling a manifold, the polyhedral domains are developed by some reflection group. In this paper we restrict our attention to domains which are cube-shaped, rather than more general reflection domains, because cubes are what are necessary to prove the main theorem.
  
  We adapt Gromov and Lawson's method, and double a cube along each face successively to get a torus. We do this by expressing a cube as the intersection of $2n$ domains with smooth boundary. We double along the first such domain, then we double along the second such domain as a subdomain of the first doubling, then double along the third such domain as a subdomain of the second doubling, and so on.
  
  The result of this procedure is a torus, consisting of $2^{2n}$ copies of our original cube. The copies of the cube have their metric modified near the boundary so that they join one another smoothly. If our original cube has nonnegative scalar curvature, strictly mean convex faces, and strictly acute dihedral angles, then the metric on the resulting torus has nonnegative scalar curvature and some points of positive scalar curvature. This contradicts the Geroch conjecture, showing that such a cube cannot exist.
  
  In Section \ref{Sec:DoubleSmooth}, we explain Gromov and Lawson's doubling trick when the boundary is a smooth manifold. In Section \ref{Sec:DoubleCube} we explain the modification when we double a cube.
  
\section{Doubling with smooth boundary}\label{Sec:DoubleSmooth}
  The contents of this section contain ideas from Gromov--Lawson 1980. We also correct a very minor computation error of theirs.
  
  \begin{thm}[Theorem 5.7 in \cite{gromov1980spin}]\label{Thm:DoubleSmooth}
    Let $X$ be a Riemannian manifold with nonnegative scalar curvature and strictly mean convex boundary $\partial X$. Then the double of $X$ (two copies of $X$ glued together along their common boundary) has a metric of nonnegative scalar curvature (and has some points with positive scalar curvature).
  \end{thm}
  \begin{proof}
    Let $X'\subset X$ be a subdomain of $X$ formed by shrinking the boundary of $X$ inward a tiny bit, but keeping $\partial X'$ still mean convex. Here is how we construct a metric of positive scalar curvature on the double of $X$.
    
    Consider $X'\times\{0\}$ as a submanifold of the Riemannian manifold $X\times\R$ with the product metric. Let $D(X)$ be the boundary of an $\varepsilon$--neighborhood of $X'\times\{0\}$. See Figure \ref{Fig:DX}. Now $D(X)$ is homeomorphic to the double of $X$. Moreover it has the same metric as $X$ on the set $X'\times\{\pm\varepsilon\}$. We claim that it has positive scalar curvature on the rounded off portions (ie away from $X'\times\{\pm\varepsilon\}$) when $\varepsilon$ is small enough.
    \begin{figure}[htbp]
      \includegraphics[page=1]{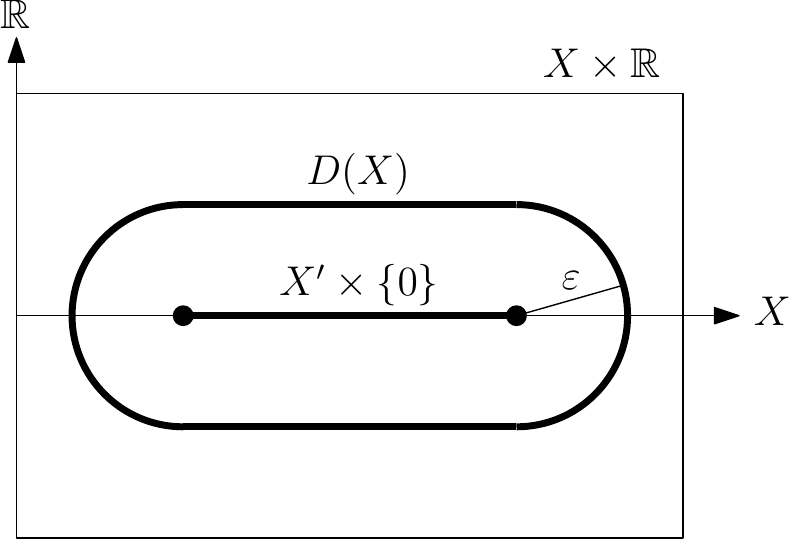}
      \centering
      \caption{$X'$ is a small inward deformation of $X$, and $D(X)$, the double of $X$, is an $\varepsilon$--neighborhood of $X'\times\{0\}$ in $X\times\R$.}
      \label{Fig:DX}
    \end{figure}
    
    Let $p\in D(X)$ lie at angle $\theta\in[0,\pi/2)$ from the horizontal. More precisely, the distance from the $X$--component of $p$ to $X'$ in $X$ is $\sqrt{\varepsilon^2-t^2}$, the $\R$--component of $p$ is $t$, and $\sin\theta=t/\varepsilon$. See Figure \ref{Fig:pDXtheta}. We compute the scalar curvature at $p$ of $D(X)$. To do that, we need to compute sectional curvatures of $D(X)$, and we use the Gauss formula for the sectional curvature of a submanifold of $X\times\R$.
    
    To compute sectional curvatures of $D(X)$, we compute its principal curvatures. Let $\gamma$ be the unit vector field along the geodesics formed by outward normal exponentiation from $\partial X'$. Let $N$ be the outward unit normal to $D(X)$. Then at $p$, $N$ is given by $N=\cos\theta\gamma + \sin\theta\ddt$. The shape operator of $D(X)$ at $p$ splits as $\cos\theta$ times the shape operator of the $\varepsilon\cos\theta$--expansion of $\partial X'$, and the shape operator of the circle of radius $\varepsilon$ in the $\theta$ direction.
    \begin{figure}[htbp]
      \includegraphics[page=2]{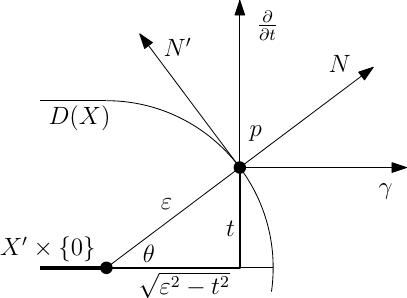}
      \centering
      \caption{A point $p$ in $D(X)$.}
      \label{Fig:pDXtheta}
    \end{figure}
    
    Therefore, if the principal curvatures of $\partial X'$ in $X$ are given by $\mu_1,\dots,\mu_{n-1}$, then the principal curvatures of $D(X)$ are given by $\lambda_0=1/\varepsilon$ in the direction $N'=-\sin\theta\gamma+\cos\theta \ddt$, and $\lambda_i=\cos\theta(\mu_i + O(\varepsilon))$ for $i=1,\dots,n-1$. The latter are in directions orthogonal to the $\R$ factor. Here $O(\varepsilon)$ represents a quantity less than a constant multiple of $\varepsilon$ as $\varepsilon$ gets arbitrarily small, and this expression follows from the fact that the metric is $C^1$; the principal curvatures don't change much as we move from $\partial X'$ to a small outward expansion of $\partial X'$.
    
    The Gauss equation tells us that the sectional curvature of a hypersurface, in directions corresponding to principal curvatures, is the ambient sectional curvature plus the product of the principal curvatures.
    
    For sectional curvatures of $D(X)$ in the direction of two principal curvatures tangent to $X$, we compute
    \[ K_{D(X)}(e_i,e_j)=K_X(e_i,e_j)+\cos^2\theta(\mu_i\mu_j+O(\varepsilon)).\]
    For sectional curvatures of $D(X)$ with one direction $N'$ and one direction tangent to $X$, we compute (using the fact that $X\times\R$ has the product metric)
    \begin{align*}
      K_{D(X)}(N',e_i) &= K_X(-\sin\theta\gamma,e_i) + \frac{1}{\varepsilon}\cos\theta(\mu_i+O(\varepsilon)) \\
      &= \sin^2\theta K_X(\gamma,e_i)+ \frac{1}{\varepsilon}\cos\theta(\mu_i+O(\varepsilon)) \\
      &= (1-\cos^2\theta) K_X(\gamma,e_i)+ \frac{1}{\varepsilon}\cos\theta(\mu_i+O(\varepsilon)) \\
      &= K_X(\gamma,e_i) + \cos\theta\left(\frac{1}{\varepsilon}\mu_i + O(1)\right).
    \end{align*}
    Now we compute the scalar curvature of $D(X)$ by adding together the sectional curvatures:
    \[ \kappa_{D(X)}=\kappa_X + \cos\theta \left( \frac{1}{\varepsilon}H + O(1)\right) \]
    Here $H=\sum \mu_i$ is the mean curvature of $\partial X'$, which we assume to be positive. Therefore for small enough $\varepsilon$, we have $\frac{1}{\varepsilon}H + O(1)>0$, and since we assume that $\kappa_X \geq 0$, we have $\kappa_{D(X)} \geq 0$.
  \end{proof}
  
\section{Doubling a cube}\label{Sec:DoubleCube}
  Performing the process from Section \ref{Sec:DoubleSmooth} on a cube will lead to a manifold which is topologically a sphere. We need to alter the procedure in a way that leads to a smooth metric on a torus. The process from the previous section changes the metric of $X$ on a small neighborhood of the boundary, in such a way that the boundary becomes totally geodesic. Consider a cube $X^n$ which is embedded inside a larger Riemannian manifold $\overline{X}^n$, and suppose that $X$ is the intersection of disks $X_1,\dots,X_{2n}$ with smooth boundary. The disks $X_i$ are chosen so that each boundary $\partial X_i$ contains one of the faces of $X$. See Figure \ref{Fig:CubeFromDisks}.
  \begin{figure}[htpb]
    \centering
    \includegraphics[page=3]{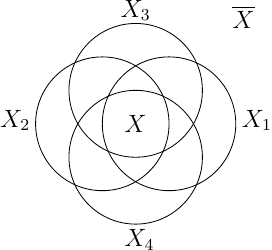}
    \caption{The $n$--dimensional cube $X\subset \overline{X}$ as an intersection of $2n$ disks $X_1,\dots,X_{2n}$, here with $n=2$.}
    \label{Fig:CubeFromDisks}
  \end{figure}
  
  We modify the notation of the last section slightly. If $X$ is a subdomain of $\overline{X}$, let $D(X,\overline{X})$ be the boundary of an $\varepsilon$-neighborhood of $X\times\{0\}$ in $\overline{X}\times\R$.
  
  We will double along each $X_i$ as follows. Let $Y_1=D(X_1,\overline{X})$. Now replace $X_j$, for $j=2,\dots,2n$, by $_1X_j=(X_j\times\R)\cap Y_1$. Now repeat, so that inductively we have
  \[ Y_{i+1} = D(_iX_{i+1}, Y_i), \]
  and again we replace the submanifolds $_iX_j$, for $j=i+2,\dots,2n$, by $_{i+1}X_j = (_iX_j\times\R)\cap Y_{i+1}$. Similarly, let $_1X$ be the subset $\bigcap_{j=2}^{2n} {_1X_j}$, and let $_iX$ be $\bigcap_{j=i+1}^{2n} {_iX_j}$. See Figure \ref{Fig:DoubleSquare}.
  \begin{figure}[htbp]
    \centering
    \includegraphics[page=4]{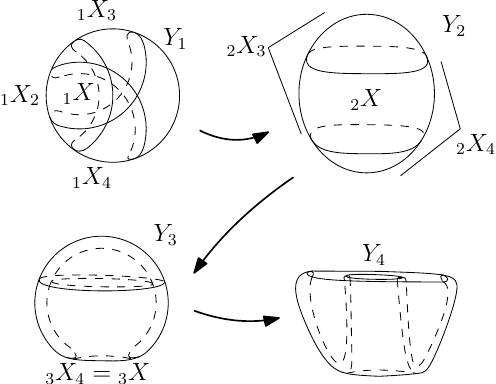}
    \caption{Doubling a square four times; the fourth stage $Y_4$ is diffeomorphic to a torus.}
    \label{Fig:DoubleSquare}
  \end{figure}
  
  The result of doubling along each face, or $Y_{2n}$, is topologically a torus. It contains $2^{2n}$ copies of $X$, doubling the number of copies with each inductive doubling step. The metric on these copies of $X$ is identical to the original metric away from the boundary, and is modified so that the boundaries between the copies is totally geodesic.
  
  We now prove the following:
  \begin{thm}\label{Thm:DoubleCube}
    Let $X$ be a cube as described above, and satisfying the following properties:
    \begin{enumerate}
      \item $X$ has nonnegative scalar curvature.
      \item Each face of $X$ is strictly mean convex.
      \item The dihedral angles between the faces of $X$ are strictly acute.
    \end{enumerate}
    Then the torus $Y_{2n}$ has nonnegative scalar curvature, and has positive scalar curvature at some points.
  \end{thm}
  It follows immediately from the Geroch Conjecture that no cube satisfying these hypotheses can exist.
  \begin{proof}
    We only have to check that the hypotheses of Theorem \ref{Thm:DoubleSmooth} are satisfied at each inductive step. In other words, we need to check that if the boundary of $_iX_j$ is mean convex as a face of $_iX$, then the boundary of $_{i+1}X_j$ is mean convex as a face of $_{i+1}X$, and that if $_iX_j$ and $_iX_k$ have an acute dihedral angle then so do $_{i+1}X_j$ and $_{i+1}X_k$. The conclusion of Theorem \ref{Thm:DoubleSmooth} ensures that nonnegative scalar curvature is preserved between inductive steps. A point of positive curvature is created in the first step and is preserved.
    
    For simplicity of notation, we check only for the first doubling: we show that if the boundary of $X_j$ is mean convex as a face of $X$, then so is the boundary of $_1X_j$ as a face of $_1X\subseteq Y_1$, and we show that if the dihedral angle between $X_j$ and $X_k$ is strictly acute, then so is the angle between $_1X_j$ and $_1X_k$. The other inductive steps are identical.
    
    Recall that $D(X_1,\overline{X})$ is the boundary of the $\varepsilon$--neighborhood of $X_1\times\{0\}$ in $\overline{X}\times\R$. Let $\tilde{p}=(p,t)\in D(X_1,\overline{X})\subseteq \overline{X}\times\R$ lie at angle $\theta$ from the horizontal in $\overline{X}\times\R$, ie $\sin\theta=t/\varepsilon$. That means that $p\in\overline{X}$ lies at distance $\varepsilon\cos\theta$ from $\partial X_1$.
    
    Suppose further that $p\in\partial X_2$, so that $\tilde{p}\in \partial _1X_2$. We will show that if the mean curvatures of $\partial X_1$ and $\partial X_2$ are positive, and $\varepsilon$ is small enough, then the mean curvature of $\partial _1X_2$ is positive too. See Figure \ref{Fig:tildepDX1}.
    \begin{figure}[htpb]
      \includegraphics[page=5]{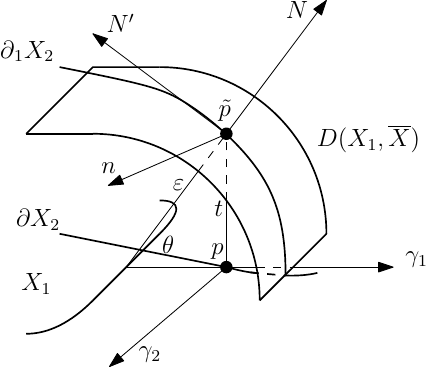}
      \centering
      \caption{A point $\tilde{p}$ in $D(X_1,\overline{X})$.}
      \label{Fig:tildepDX1}
    \end{figure}
    
    Let $\gamma_i$ be the unit vector field tangent to outward geodesics normal to $\partial X_i$, $i=1,2$. These are defined in a neighborhood of $\partial X_i$. For $\varepsilon$ small enough, they are defined at $p$. Moreover, when dihedral angles are acute, $\langle \gamma_1,\gamma_2 \rangle < 0$ along the intersection of the two boundaries, and the inequality holds also at $p$ for small enough $\varepsilon$. We use $\langle \cdot,\cdot \rangle$ to denote the metric both on $\overline{X}$ and on $\overline{X}\times\R$ and rely on context to distinguish the two.
    
    Let $N$ be the outward unit normal to $Y_1 = D(X_1,\overline{X})$:
    \[ N = \cos\theta \gamma_1 + \sin\theta \ddt \]
    Let $\pi:T_{\tilde{p}}Y_1\to T_p \overline{X}$ be projection to the first factor, ie $\pi(v,c\ddt) = v$. A straightforward computation shows that $\pi^{-1}(v) = v - \cot\theta \ddt$. (We need only compute that $\pi^{-1}(v)$ has the correct $t$--component to make $\langle \pi^{-1}(v),N \rangle = 0$.)
    
    Let $\pi^*:T_p \overline{X}\to T_{\tilde{p}}Y_1$ be the adjoint of $\pi$. We will make several uses of this map in the computations that follow. We compute
    \begin{align*}
      \pi^*(v) &= v - \cos^2\theta \langle v,\gamma_1 \rangle\gamma_1 - \cos\theta\sin\theta \langle v, \gamma_1 \rangle \ddt \\ 
      &= v - \cos\theta \langle v, \gamma_1 \rangle N
    \end{align*}
    From this we compute
    \[ \pi\circ\pi^*(v) = v - \cos^2\theta \langle v,\gamma_1 \rangle \gamma_1, \]
    and from that we have
    \[ \pi\circ\pi^* = I - \cos^2\theta \langle \cdot, \gamma_1 \rangle \gamma_1. \]
    
    We now compute the mean curvature of $\partial _1X_2$ at $\tilde{p}$. Note that $\gamma_2$ is the outward normal to $\partial X_2$ at $p$, and hence $\pi^*(\gamma_2)$ is an outward normal to $\partial _1X_2$ at $\tilde{p}$. That is because tangent vectors to $\partial _1X_2$ are of the form $\pi^{-1}(v)$ for $v\in T_p(\partial X_2)$, and we have
    \begin{align*}
      \langle \pi^*(\gamma_2) , \pi^{-1}(v) \rangle &= \langle \gamma_2, \pi\pi^{-1}(v) \rangle \\
      &= \langle \gamma_2, v \rangle \\
      &= 0.
    \end{align*}
    Thus $\pi^*(\gamma_2)$ is orthogonal to all tangent vectors of $T_{\tilde{p}}(\partial _1X_2)$, and we denote this vector by $n$.
    
    We take an arbitrary orthonormal basis $\{e_i\}$ of $T_{\tilde{p}}(\partial _1X_2)$ and compute the mean curvature as:
    \[ H_{\tilde{p}} = \sum \left\langle \nabla_{e_i} \frac{n}{|n|}, e_i \right\rangle \]
    Here the connection is the Levi-Civita connection on $Y_1$, but equivalently we  use the connection on $\overline{X}\times \R$ since we are taking the inner product with $e_i$ which is tangent to $Y_1$.
    
    From $\langle n, e_i \rangle = 0$ we have
    \[ H_{\tilde{p}} = \frac{1}{|n|}\sum \langle \nabla_{e_i} n, e_i \rangle. \]
    Combined with our expression for $n=\pi^*(\gamma_2)$, we have an expression for a positive multiple of $H_{\tilde{p}}$:
    \begin{align*}
      |n|\ H_{\tilde{p}} &= \sum \langle \nabla_{e_i} \gamma_2, e_i \rangle \ + \  \sum \langle \nabla_{e_i}(-\cos\theta \langle \gamma_2, \gamma_1 \rangle N), e_i \rangle \\
      &= \mathbf{(1)} + \mathbf{(2)}
    \end{align*}
    We now evaluate (1) and (2) above separately, and show that their sum is positive when $\varepsilon$ is small enough.
    
    First we evaluate (1). Let $S:T_p X_2 \to T_p X_2$ be the shape operator
    \[ S(v) = \nabla_v \gamma_2, \]
    and note that the mean curvature $H_p$ of $X_2$ at $p$ is given by $\tr S$. Note that $S(v)$ is well-defined for $v\in T_p X_2$ if we take the tangential component afterward. (1) is almost but not quite equal to $\tr S$; the set $\{e_i\}$ is an orthonormal basis of $T_{\tilde{p}}(_1X_2)$, not of $T_pX_2$. Because $\nabla$ is the Levi-Civita connection associated with the product metric on $\overline{X}\times \R$, we have $\nabla_{\ddt} \gamma_2 = 0$, and also $\nabla_{e_i}\gamma_2$ has no $t$--component. Therefore we have
    \begin{align*}
      \mathbf{(1)} &= \sum \langle \nabla_{e_i} \gamma_2, e_i \rangle \\
      &= \sum \langle \nabla_{\pi(e_i)} \gamma_2, \pi(e_i) \rangle \\
      &= \sum \langle \pi^*(\nabla_{\pi(e_i)} \gamma_2), e_i \rangle \\
      &= \tr (\pi^*\circ S \circ \pi) \\
      &= \tr (S \circ \pi \circ \pi^* ) \\
      &= \tr (S \circ (I - \cos^2\theta \langle \cdot, \gamma_1 \rangle \gamma_1)) \\
      &= \tr(S) - \cos^2\theta \langle \nabla_{\gamma_1}\gamma_2, \gamma_1 \rangle \\
      &= H_p - \cos^2\theta \langle \nabla_{\gamma_1}\gamma_2, \gamma_1 \rangle.
    \end{align*}
    
    Now we evaluate (2). From $\langle N, e_i \rangle = 0$ we have
    \begin{align*}
      \mathbf{(2)} &= \sum \langle \nabla_{e_i}(-\cos\theta \langle \gamma_2, \gamma_1 \rangle N), e_i \rangle \\
      &= -\cos\theta \langle \gamma_2, \gamma_1 \rangle \sum \langle \nabla_{e_i} N, e_i \rangle.
    \end{align*}
    We now observe that $\sum \langle \nabla_{e_i} N, e_i \rangle$ is very nearly the mean curvature $H_{\tilde{p},Y_1}$ at $\tilde{p}$ of $Y_1$ in $\overline{X}\times \R$. An orthonormal basis for $T_{\tilde{p}}Y_1$ is $\{e_i, n/|n|\}$, so in fact we have:
    \begin{align*}
      \mathbf{(2)} &= -\cos\theta \langle \gamma_2, \gamma_1 \rangle ( H_{\tilde{p},Y_1} - \langle \nabla_{n/|n|} N, n/|n| \rangle )
    \end{align*}
    From the computations in the  previous section, we have
    \[ H_{\tilde{p},Y_1} = \frac{1}{\varepsilon}+\cos\theta(H_{p,X_1}+O(\varepsilon)), \]
    where $H_{p,X_1}$ is the mean curvature of the boundary of $X_1$ in $\overline{X}$ at the point nearest $p$. Moreover, $\langle \nabla_{n/|n|} N, n/|n| \rangle = \langle S_{Y_1}(n/|n|), n/|n| \rangle$, where $S_{Y_1}$ is the shape operator of $Y_1$ at $\tilde{p}$. Decompose $n/|n|$ into its components in the $\frac{1}{\varepsilon}$--eigenvector direction and the complementary directions. The unit $\frac{1}{\varepsilon}$--eigenvector is $N'=-\sin\theta\gamma_1 + \cos\theta\ddt$, so we write
    \[ \frac{n}{|n|} = \langle n/|n|, N' \rangle N' + v, \]
    where $v\in T_{\tilde{p}}Y_1$ is orthogonal to $N'$ and $|v|\leq 1$. Then we have
    \begin{align*}
      \langle \nabla_{n/|n|} N, n/|n| \rangle &= \langle S_{Y_1}(n/|n|), n/|n| \rangle \\
      &= \langle S_{Y_1}(\langle n/|n|, N' \rangle N' + v), n/|n| \rangle \\
      &= \langle n/|n|, N' \rangle \langle S_{Y_1}(N'), n/|n| \rangle + \langle S_{Y_1}(v), n/|n| \rangle \\
      &= \langle n/|n|, N' \rangle^2 \frac{1}{\varepsilon} + \cos\theta M,
    \end{align*}
    where $M$ is some quantity less than $|S_{\partial X_1}| + O(\varepsilon)$; that's because $v$ is tangent to the normal $\varepsilon$--deformation of $\partial X_1$, and hence $S_{Y_1}(v)$ can be only as large as the norm of the shape operator of $\partial X_1$ (plus $\varepsilon$).
    
    Putting the pieces together, we have
    \begin{align*}
      \mathbf{(2)} &= -\cos\theta \langle \gamma_2, \gamma_1 \rangle ( H_{\tilde{p},Y_1} - \langle \nabla_{n/|n|} N, n/|n| \rangle ) \\
      &= -\cos\theta \langle \gamma_2, \gamma_1 \rangle \left( \frac{1}{\varepsilon}+\cos\theta(H_{p,X_1}+O(\varepsilon)) - \langle n/|n|, N' \rangle^2 \frac{1}{\varepsilon} - \cos\theta M \right) \\
      &= -\cos\theta \langle \gamma_2, \gamma_1 \rangle \left( \frac{1}{\varepsilon}(1-
      \langle n/|n|, N' \rangle^2) + \cos\theta\ O(1)\right) \\
      &= -\cos\theta \langle \gamma_2,\gamma_1 \rangle \left( C\frac{1}{\varepsilon} + \cos\theta\ O(1)\right),
    \end{align*}
    where $C\in[0,1]$ by the Cauchy-Schwarz inequality. In fact we have
    \begin{align*}
      C &= 1 - \langle n/|n|, N' \rangle^2 \\
      &=1-\left\langle\frac{\gamma_2-\cos\theta\langle\gamma_2,\gamma_1\rangle N}{|\gamma_2-\cos\theta\langle\gamma_2,\gamma_1\rangle N|},N'\right\rangle^2\\
      &= 1 - \frac{\langle \gamma_2,N' \rangle^2}{|\gamma_2-\cos\theta\langle\gamma_2, \gamma_1\rangle N|^2}\\
      &= 1 - \frac{\sin^2\theta \langle \gamma_2,\gamma_1 \rangle^2}{1-\cos^2\theta \langle \gamma_2,\gamma_1 \rangle^2}\\
      &= \frac{1-\langle \gamma_2,\gamma_1 \rangle^2}{1-\cos^2\theta \langle \gamma_2,\gamma_1 \rangle^2},
    \end{align*}
    which is strictly greater than 0 (independent of $\theta$) if we're assuming that $\partial X_1$ and $\partial X_2$ are transverse (and hence that $\langle \gamma_2,\gamma_1 \rangle \neq \pm 1$). Therefore $C > 0$.
    
    Finally, we have
    \begin{align*}
      |n|H_{\tilde{p}} &= \mathbf{(1)}+\mathbf{(2)} \\
      &= H_p - \cos^2\theta \langle \nabla_{\gamma_1}\gamma_2, \gamma_1 \rangle -\cos\theta \langle \gamma_2,\gamma_1 \rangle \left( C\frac{1}{\varepsilon} + \cos\theta\ O(1)\right) \\
      &= H_p + \cos\theta \left( -\langle \gamma_2,\gamma_1 \rangle C \frac{1}{\varepsilon} + \cos\theta\ O(1) \right).
    \end{align*}
    We assume that the dihedral angle between $\partial X_1$ and $\partial X_2$ is acute, and therefore $\langle \gamma_2,\gamma_1 \rangle < 0$. Therefore for small enough $\varepsilon$, we have $-\langle \gamma_2,\gamma_1 \rangle C \frac{1}{\varepsilon} + \cos\theta\ O(1) > 0$, and since we also assume $H_p > 0$, we have $H_{\tilde{p}} > 0$.
    
    It only remains to show that if $X_2$ and $X_3$ have an acute dihedral angle, then so do $_1X_2$ and $_1X_3$. We assume $\langle \gamma_2,\gamma_3 \rangle < 0$, and wish to show that the normal vectors to $_1X_2$ and $_1X_3$ have inner product less than zero. Those normal vectors are given by $\pi^*(\gamma_2)$ and $\pi^*(\gamma_3)$ respectively. We have:
    \begin{align*}
      \langle \pi^*(\gamma_2), \pi^*(\gamma_3) \rangle &= \langle \gamma_2, \pi\pi^*(\gamma_3) \rangle \\
      &= \langle \gamma_2, \gamma_3 - \cos^2\theta \langle \gamma_3, \gamma_1 \rangle \gamma_1 \rangle \\
      &= \langle \gamma_2,\gamma_3 \rangle - \cos^2\theta \langle \gamma_3,\gamma_1 \rangle \langle \gamma_2,\gamma_1 \rangle \\
      &\leq \langle \gamma_2, \gamma_3 \rangle \\
      &< 0.
    \end{align*}
    That concludes the proof.
  \end{proof}

\bibliographystyle{plain}
\bibliography{cubedoubling.bib}

\begin{thebibliography}{1}

\bibitem{gromov1980spin}
Mikhael Gromov and H~Blaine Lawson.
\newblock Spin and scalar curvature in the presence of a fundamental group.
  {I}.
\newblock {\em Annals of Mathematics}, pages 209--230, 1980.

\bibitem{gromov2014dirac}
Misha Gromov.
\newblock Dirac and plateau billiards in domains with corners.
\newblock {\em Central European Journal of Mathematics}, 12(8):1109--1156,
  2014.

\end{thebibliography}

\end{document}